\documentclass[11pt]{article}
\usepackage [dvips]{graphics}
\usepackage[centertags]{amsmath}
\usepackage{amsfonts}
\usepackage{amssymb}
\usepackage{amsthm}
\usepackage{newlfont}
\usepackage{stmaryrd}
\usepackage[]{titletoc}
\usepackage{lipsum}
\usepackage{soul}
\newlength{\defbaselineskip}
\newcommand{\setlinespacing}[1]%
           {\setlength{\baselineskip}{#1 \defbaselineskip}}

\theoremstyle{plain}
\newtheorem{thm}{Theorem}[section]
\newtheorem{defn}[thm]{Definition}
\newtheorem{cor}[thm]{Corollary}
\newtheorem{lem}[thm]{Lemma}

\newtheorem{exam}[thm]{Example}
\newtheorem{rem}[thm]{Remark}

\newtheorem*{arv}{Arveson-Douglas Conjecture}
\newtheorem*{geo}{Geometric Arveson-Douglas Conjecture}

\titlecontents{section}[2em]{}
    {\thecontentslabel\quad}{}{{}}

\titlecontents{subsection}[2em]{}
    {\thecontentslabel\quad}{}{{}}

\newcommand{\bn}{\mathbb{B}_n}
\newcommand{\cn}{\mathbb{C}^n}
\newcommand{\toe}{\mathcal{T}(L^{\infty})}
\newcommand{\ber}{L_a^{2}(\mathbb{B}_n)}
\newcommand{\pbn}{\partial\mathbb{B}_n}
\newcommand{\clb}{\overline{\mathbb{B}_n}}

\newcommand\blfootnote[1]{%
  \begingroup
  \renewcommand\thefootnote{}\footnote{#1}%
  \addtocounter{footnote}{-1}%
  \endgroup
}

\makeatletter\@addtoreset{equation}{section} \makeatother
\begin{document}
\title{Geometric Arveson-Douglas Conjecture-Decomposition of Varieties}
\author{Ronald G. Douglas, Yi Wang}
\maketitle
\begin{abstract}
In this paper, we prove the Geometric Arveson-Douglas Conjecture for a special case which allow some singularity on $\partial{\mathbb{B}_n}$. More precisely, we show that if a variety can be decomposed into two varieties, each having nice properties and intersecting nicely with $\partial\mathbb{B}_n$, then the Geometric Arveson-Douglas Conjecture holds on this variety. We obtain this result by applying a result by Su\'{a}rez, which allows us to ``localize'' the problem. Our result then follows from the simple case when the two varieties are intersection of linear subspaces with $\bn$.
\end{abstract}
\section{Introduction}
A classical way of obtaining submodules of the Bergman space $\ber$ is from a zero variety in $\bn$. The corresponding quotient module is then just the closure of linear span of those reproducing kernels in this variety. The Geometric Arveson-Douglas Conjecture asks whether the quotient module is $(p-)$ essential normal. In this paper, we develop a machinery to tell whether the sum of two quotient modules, each related to a zero variety in $\bn$, is closed. In other words, we prove that when two varieties satisfy certain hypotheses, the quotient module related to their union, which is itself a zero variety, is exactly the sum of the two quotient modules. Therefore properties of the union variety is determined by properties of the two original ones. This ``decomposing variety'' technique allows us to obtain complicated examples from simple ones (cf. \cite{Douglas Wang remark}\cite{Sha Ken}\cite{Shalit}). As a consequence, under certain hypotheses, the essential normality results on the union variety follow from those on each variety.

A complex Hilbert space $\mathcal{H}$ is called a Hilbert module (over the polynomial ring $\mathbb{C}[z_1,\cdots,z_n]$) if for every $p\in\mathbb{C}[z_1,\cdots,z_n]$ there is a bounded linear operator $M_p$ on $\mathcal{H}$ and the map $\mathbb{C}[z_1,\cdots,z_n]\to B(\mathcal{H}),~p\to M_p$ is an algebra homomorphism. A submodule $P\subset\mathcal{H}$ is a Hilbert subspace of $\mathcal{H}$ that is closed under the module multiplications $M_p$. Let $Q$ be the orthogonal complement of $P$ in $\mathcal{H}$. Then $Q$ is the quotient Hilbert module with the homomorphism taking $z_i$ to
the compression of $M_{z_i}$ to $Q$. A Hilbert module $\mathcal{H}$ is said to be essentially normal ($p$-essentially normal) if the commutators $[M_{z_i}, M_{z_j}^*]$ belong to the compact operators $\mathcal{K}(\mathcal{H})$ (Schatten $p$ class $\mathcal{S}^p$), for any $1\leq i, j\leq n$.

In the study of spaces of holomorphic functions on the open unit ball $\bn$, for example, the Bergman space, the Hardy space and the Drury-Arveson space, a well known property is that they are essentially normal, i.e., the multiplication operators of analytic polynomials are essentially normal.
In his paper \cite{Arv Dirac}, Arveson made a conjecture about essential normality of submodules and quotient modules of the Drury-Arveson space. The conjecture was then refined by the first author and extended to other spaces \cite{Dou index}. In this paper, we consider the Bergman space.
\begin{arv}
Assume $I$ is a homogeneous ideal of the polynomial ring $\mathbb{C}[z_1,\cdots,z_n]$
and $[I]$ is the closure of $I$ in $\ber$. Then for all $p>\dim Z(I)$, the quotient module $Q=[I]^{\perp}$ is $p$-essentially normal. Here
$$
Z(I)=\{z\in\bn:p(z)=0, \forall p\in I\}
$$
and $\dim Z(I)$ denotes the complex dimension of $Z(I)$.
\end{arv}
Results on the Arveson-Douglas Conjecture include \cite{Arv p summable}\cite{Douglas Sarkar}\cite{Douglas Wang}\cite{Guo}$\sim$\cite{Shalit} and many others.

For a polynomial ideal $I$, the zero variety $Z(I)$ determines another submodule
$$
P=\{f\in\ber: f|_{Z(I)}=0\}.
$$
In many cases, especially when $I$ is radical, one can prove that $[I]=P$. Specializing to this case, we have a geometric version of the conjecture \cite{Sha Ken}.
\begin{geo}
Let $M$ be a homogeneous variety in $\bn$. Let
$$
P=\{f\in\ber: f|_M=0\}
$$
and $Q=P^{\perp}=\overline{span}\{K_{\lambda}:\lambda\in M\}$. Then the quotient module $Q$ is $p$-essentially normal for every $p>\dim M$.
\end{geo}
Although the Geometric Arveson-Douglas Conjecture is about homogenous varieties, it was shown to hold in many non-homogenous cases. Assuming $M$ is smooth on $\pbn$, the conjecture was proved by Engli\v{s} and Eschmeier \cite{Englis}, the first author, Tang and Yu \cite{DYT} and us \cite{our paper} under additional assumptions, using very different techniques. In particular, in \cite{our paper}, we introduced tools in complex harmonic analysis to solve the problem. One approach is to decompose the variety into simpler ones. In \cite{Sha Ken}, Shalit and Kenndy discussed this problem and obtained positive results on unions of linear varieties or homogenous varieties intersecting only at the origin. In this paper, we continue to develop the machinery in \cite{our paper} and discuss a simple case which allow some singularity on $\pbn$. Suppose $M_1$ and $M_2$ are two varieties, each having nice properties, the variety $M=M_1\cup M_2$ might have singular points on $\pbn$. In this paper, we prove the Geometric Arveson-Douglas Conjecture on $M$ assuming the varieties $M_1$ and $M_2$ intersect ``nicely''. Our main result is the following.
\begin{thm}\label{main theorem}
Suppose $\tilde{M_1}$ and $\tilde{M_2}$ are two analytic subsets of an open neighborhood of $\clb$. Let $\tilde{M_3}=\tilde{M_1}\cap\tilde{M_2}$. Assume that
\begin{itemize}
\item[(i)]$\tilde{M_1}$ and $\tilde{M_2}$ intersect transversely with $\pbn$ and have no singular points on $\pbn$.
\item[(ii)]$\tilde{M_3}$ also intersects transversely with $\pbn$ and has no singular points on $\pbn$.
\item[(iii)]$\tilde{M_1}$ and $\tilde{M_2}$ intersect cleanly on $\partial\mathbb{B}_n$.
\end{itemize}
Let $M_i=\tilde{M_i}\cap\mathbb{B}_n$ and $Q_i=\overline{span}\{K_{\lambda}: \lambda\in M_i\}$, $i=1,2,3.$ Then $Q_1\cap Q_2/Q_3$ is finite dimensional and $Q_1+Q_2$ is closed. As a consequence, $Q_1+Q_2$ is $p$-essentially normal for $p>2d$, where $d=\max\{\dim M_1, \dim M_2\}$.
\end{thm}
A key ingredient is that when the varieties $M_1$ and $M_2$ satisfy condition (i), the projections $Q_1$ and $Q_2$ are in the Toeplitz algebra $\toe$ (cf. \cite{our paper}).  This allows us to apply a result by Su\'{a}rez (cf. \cite{Suarez07}) about essential norms of Toeplitz operators. Then we show that the angle between the two spaces $Q_1$ and $Q_2$ depend essentially on the angles between $M_1$ and $M_2$ at points in $\pbn$. Therefore under the hypotheses of Theorem \ref{main theorem}, $Q_1$ and $Q_2$ have positive angle, which implies that their sum is closed. As a consequence, an equality in index classes is stated in the Summary.

Our method offers a framework of proving closeness of sum of spaces. This could be useful not only in this paper, but also in the future.

We would like to thank Xiang Tang for discussing with us, reading the draft of this paper and giving valuable suggestions and advises.

\section{Preliminary}
To simplify notation, throughout this paper, we will use the same letter to denote both the space and the projection operator onto it. For example, the letter $Q$ is used to denote both the quotient module and the projection operator onto $Q$.

For $n\in\mathbb{N}$, let $\bn$ be the open unit ball in $\cn$. Let $\ber$ be the Bergman space on $\bn$.
$$
\ber=\{f\mbox{ holomorphic on }\bn: \int_{\bn}|f|^2dv<\infty\}.
$$
Here $v$ is the normalized Lebesgue measure on $\bn$, $v(\bn)=1$. The Bergman space $\ber$ has reproducing kernels
$$
K_z(w)=\frac{1}{(1-\langle w,z\rangle)^{n+1}}, ~~~~~~z,w\in\bn.
$$
It becomes a Hilbert module with module map given by pointwise multiplication.
For $g\in L^{\infty}(\bn)$, the Toeplitz operator is defined by
$$
T_g:\ber\to\ber, f\mapsto P_{\ber}(gf)
$$
where $P_{\ber}$ is the projection operator from $L^2(\bn)$ to $\ber$. The Toeplitz algebra $\toe$ is the $C^*$ subalgebra of $B(\ber)$ generated by $T_g$, $g\in L^{\infty}(\bn)$.
\begin{defn}
Let $\Omega$ be a complex manifold. A set $A\subset\Omega$ is called a \emph{(complex) analytic subset} of $\Omega$ if for each point $a\in\Omega$ there exist a neighborhood $U\ni a$ and functions $f_1,\cdots,f_N$ holomorphic on this neighborhood such that
$$
A\cap U=\{z\in U: f_1(z)=\cdots=f_N(z)=0\}.
$$
A point $a\in A$ is called \emph{regular} if there is a neighborhood $U\ni a$ in $\Omega$ such that $A\cap U$ is a complex submanifold of $\Omega$. A point $a\in A$ is called a \emph{singular point} of $A$ if it is not regular.
\end{defn}
\begin{defn}
Let $Y$ be a manifold and let $X, Z$ be two submanifolds of $Y$. We say that the submanifolds $X$ and $Z$ are \emph{transversal} if $\forall x\in X\cap Z$, $T_x(X)+T_x(Z)=T_x(Y)$.
\end{defn}

The following two theorems are the main result of our paper \cite{our paper}. We will use them in the proof.
\begin{thm}
If there exists a positive, finite, regular, Borel measure~$\mu$ on $M$ such that the $L^2(\mu)$ norm and Bergman norm are equivalent on $Q$, i.e., $\exists C, c>0$ such that $\forall f\in Q$,
$$c\|f\|^2\leq\int_M|f(w)|^2d\mu(w)\leq C\|f\|^2,$$
then the projection operator $Q$ belongs to the Toeplitz algebra $\toe$ and consequently, the quotient module~$Q$~is essentially normal.
\end{thm}
In the sequel, we will say $\mu$ is an ``equivalent measure'' on $Q$ if it satisfies the above hypotheses.
\begin{thm}\label{one variety}
Suppose~$\tilde{M}$~is a complex analytic subset of an open neighborhood of~$\clb$~satisfying the following conditions:
\begin{itemize}
\item[(1)] $\tilde{M}$~intersects~$\pbn$~transversely.
\item[(2)] $\tilde{M}$~has no singular points on~$\pbn$.
\end{itemize}
Let~$M=\tilde{M}\cap\bn$ and let~$P=\{f\in\ber: f|_M=0\}$. Then $Q$ has an ``equivalent measure''. As a consequence, the projection operator $Q\in\toe$ and the quotient module~$Q$~is $p$-essentially normal for any $p>2\dim M$.
\end{thm}

Next, we introduce some tools in complex harmonic analysis that will be used in this paper (see \cite{Zhu Kehe}\cite{Suarez04}\cite{Suarez07} for more details).

\begin{defn}
For~$z\in\bn$, write~$P_z$~for the orthogonal projection onto the complex line~$\mathbb{C}z$ and~$Q_z=I-P_z$. The map
$$
\varphi_z(w)=\frac{z-P_z(w)-(1-|z|^2)^{1/2}Q_z(w)}{1-\langle w,z\rangle}
$$
defines an automorphism of $\mathbb{B}_n$.
\end{defn}
The map $\varphi_z$ has many nice properties, for example, it maps affine spaces to affine spaces. Also, $\varphi_z\circ\varphi_z=id$~and~$\varphi_z(0)=z$.

\begin{lem}\label{basic about varphi}
Suppose~$a$, $z$, $w\in\bn$, then
\begin{itemize}
\item[(1)] $$1-\langle\varphi_a(z),\varphi_a(w)\rangle=\frac{(1-\langle a,a\rangle)(1-\langle z,w\rangle)}{(1-\langle z,a\rangle)(1-\langle a,w\rangle)}.$$
\item[(2)] As a consequence of (1),
$$1-|\varphi_a(z)|^2=\frac{(1-|a|^2)(1-|z|^2)}{|1-\langle z,a\rangle|^2}.$$
\item[(3)] The Jacobian of the automorphism~$\varphi_z$~is
$$(J\varphi_z(w))=\frac{(1-|z|^2)^{n+1}}{|1-\langle w,z\rangle|^{2(n+1)}}.$$
\end{itemize}
\end{lem}

\begin{defn}
For $z, w\in\mathbb{B}_n$, define
$$
\rho(z,w)=|\varphi_z(w)|.
$$
$\rho$ is called the pseudo-hyperbolic metric on $\mathbb{B}_n$. Define
$$
\beta(z,w)=\frac{1}{2}\log\frac{1+\rho(z,w)}{1-\rho(z,w)}.
$$
$\beta$ is called the hyperbolic metric on $\bn$.
\end{defn}
From Lemma \ref{basic about varphi}, one can show that:
\begin{lem}
Fix $r>0$, then $\exists c, C>0$ such that
\begin{itemize}
\item[(1)] $c<\frac{1-|z|^2}{1-|w|^2}<C$, $\forall z, w, \beta(z,w)<r$.
\item[(2)] $c<\frac{|1-\langle z,w\rangle|}{1-|z|^2}<C$, $\forall z,w, \beta(z,w)<r$.
\end{itemize}
\end{lem}
For~$r>0$, $z\in\bn$, write
$$D(z,r)=\{w\in\bn: \beta(w,z)< r\}=\{w\in\bn: \rho(w,z)< s_r\},$$
where~$s_r=\tanh r$.

\begin{lem}{\cite[2.2.7]{Rudin}}\label{description of D(z,R)}
For~$z\in\bn$, $r>0$, the hyperbolic ball~$D(z,r)$~consists of all~$w$~that satisfy:
$$\frac{|Pw-c|^2}{s_r^2\rho^2}+\frac{|Qw|^2}{s_r^2\rho}<1,$$
where~$P=P_z$, $Q=Q_z$, and
$$c=\frac{(1-s_r^2)z}{1-s_r^2|z|^2},~~~\rho=\frac{1-|z|^2}{1-s_r^2|z|^2}.$$
\end{lem}
Thus~$D(z,r)$~is an ellipsoid with center~$c$, radius of~$s_r\rho$~in the~$z$~direction and~$s_r\sqrt{\rho}$~in the directions perpendicular to~$z$. Therefore the Lebesgue measure of~$D(z,r)$~is
$$v_n(D(z,r))=Cs_r^{2n}\rho^{n+1},$$
where~$C>0$~is a constant depending only on~$n$.
Note that when we fix~$r$, $\rho$~is comparable with~$1-|z|^2$. Hence~$v(D(z,r))$~is comparable with~$(1-|z|^2)^{n+1}$.

\begin{lem}\label{4equivalent}
Let~$\nu$~be a positive, finite, regular, Borel measure on~$\bn$~and~$r>0$. Then the following are equivalent.
When one of these conditions holds, $\nu$ is called a Carleson measure (for $\ber$).
\begin{itemize}
\item[(1)]$\sup_{z\in\bn}\int_{\bn}\frac{(1-|z|^2)^{n+1}}{|1-\langle w, z\rangle|^{2(n+1)}}d\nu(w)<\infty,$
\item[(2)]$\exists C>0:\int|f|^2d\nu\leq C\int|f|^2dv~\mbox{for all}~f\in\ber,$
\item[(3)]$\sup_{z\in\bn}\frac{\nu(D(z,~r))}{v_n(D(z,~r))}<\infty,$
\end{itemize}
\end{lem}

Let $A$ be the algebra of bounded functions on $\mathbb{B}_n$ which are uniformly continuous in the hyperbolic metric, equipped with the supreme norm. Then $A$ is a commutative $C^*$ algebra. Let $M_A$ be its maximal ideal space. Then the unit ball $\mathbb{B}_n$ is naturally contained in $M_A$ as evaluations. The algebra $A$ can be used to study the properties of the Toeplitz operators( cf. \cite{Suarez04}\cite{Suarez07}).

\begin{defn}
A sequence $\{z_m\}\subseteq\mathbb{B}_n$ is said to be \emph{separated} if there exists $\delta>0$ such that $\rho(z_k,z_l)\geq\delta$ for $k\neq l$.

If  $x, y\in M_A$, define
$$\rho(x,y)=\sup\rho(\mathcal{S},\mathcal{T}),$$
 where $\mathcal{S}$, $\mathcal{T}$ run over all separated sequences in $\mathbb{B}_n$ so that $x\in\overline{\mathcal{S}}^{_A}$ and $y\in\overline{\mathcal{T}}^{_A}$. Here $\overline{\mathcal{S}}^{_A}$ denotes the closure in $M_A$. Define
$$
\beta(x,y)=\frac{1}{2}\log\frac{1+\rho(x,y)}{1-\rho(x,y)}.
$$

\end{defn}

For $x\in M_A$ and any net $\{z_{\alpha}\}$ that converges to it, there is a map $\varphi_x:\mathbb{B}_n\to M_A$ such that $a\circ\varphi_x\in A$ and $a\circ\varphi_{z_{\alpha}}\to a\circ\varphi_x $ uniformly on compact sets of $\mathbb{B}_n$, for all $a\in A$ (cf. \cite{Suarez04}).

The following Lemma was proved in \cite{Suarez04}, Section 3.2 for the unit disc and the same proof works for the $\mathbb{B}_n$ case verbatimly.

\begin{lem}\label{extended rho}
Let $x, y\in M_A\backslash\mathbb{B}_n$. Then
\begin{itemize}
\item[(1)] $\rho(x,y)=a<1$ if and only if $y=\varphi_x(w)$ for some $w$ with $|w|=a$.
\item[(2)] $y=\varphi_x(\xi)$ with $\xi\in\mathbb{B}_n$ if and only if every separated sequences $\mathcal{S}$, $\mathcal{T}$ such that $x\in\overline{\mathcal{S}}^{_A}$ and $y\in\overline{\mathcal{T}}^{_A}$ satisfy $\rho(\mathcal{T},\{\varphi_{z_n}(\xi): z_n\in\mathcal{S}\})=0$.
\item[(3)] $\rho(\varphi_x(\xi_1),\varphi_x(\xi_2))=\rho(\xi_1,\xi_2)$ for every $\xi_1$, $\xi_2\in\mathbb{B}_n$.
\item[(4)] $\beta$ is a $[0,+\infty]$-valued metric on $M_A$.
\end{itemize}
\end{lem}

\begin{defn}
For $z\in\mathbb{B}_n$, define $U_z: L_a^2(\mathbb{B}_n)\to L_a^2(\mathbb{B}_n)$ to be
$$
U_z(f)=f\circ\varphi_z\cdot k_z.
$$
Here $k_z$ is the normalized reproducing kernel.
$$
k_z(w)=\frac{K_z(w)}{\|K_z\|}=\frac{(1-|z|^2)^{(n+1)/2}}{(1-\langle w,z\rangle)^{n+1}}.
$$
Then $U_z$ is an unitary operator on $L_a^2(\mathbb{B}_n)$ with $U_z^*=U_z$.
\end{defn}
The following Lemmas can be found in Section 8 and 10 of \cite{Suarez07}.
\begin{lem}
If $S\in\mathcal{T}(L^{\infty})$, then the map $\Psi_S:\mathbb{B}_n\to(\mathcal{B}(L_a^2(\mathbb{B}_n)), SOT)$, $z\mapsto S_z:=U_zSU_z$ extends continuously to $M_A$. We write $S_x$ for the operator $\Psi_S(x)$.
\end{lem}
\begin{lem}
$x\in M_A$, $S, T\in\mathcal{T}(L^{\infty})$, then
$$
(ST)_x=S_xT_x, (S_x)^*=(S^*)_x, \|S_x\|\leq\|S\|.
$$
\end{lem}
From the Lemma, we see that for any normal Toeplitz operator $S$ and any $f\in C(\sigma(S))$, $f(S)_x=f(S_x)$.

\begin{lem}\label{essential norm}
$S\in\mathcal{T}(L^{\infty})$, then
$$
\|S\|_e=\sup_{x\in M_A\backslash\mathbb{B}_n}\|S_x\|.
$$
\end{lem}

Suppose $H$ is a Hilbert space and $H_1$, $H_2$ are subspaces of $H$. When is $H_1+H_2$ closed? Write $H_3=H_1\cap H_2$. Then $H_1+H_2=(H_1\ominus H_3+H_2\ominus H_3)\oplus H_3$. Therefore $H_1+H_2$ is closed if and only if $H_1\ominus H_3+H_2\ominus H_3$ is closed. In the case when $H_1\cap H_2=\{0\}$, by open mapping Theorem we know that $H_1+H_2$ is closed if and only if the norm on $H_1+H_2$ is equivalent to the norm on $H_1\oplus H_2$.
\begin{defn}\label{defnangle}
Suppose $H_1$, $H_2$ are subspaces of a Hilbert space $H$. And write $H_3=H_1\cap H_2$.
We define the \emph{angle} of $H_1$ and $H_2$ to be
$$
\arccos\sup\{\frac{|\langle u,v\rangle|}{\|u\|\|v\|}: u\in H_1\ominus H_3, v\in H_2\ominus H_3\}.
$$
\end{defn}

\begin{lem}\label{angle}
The angle is positively related to the following quantities.
\begin{itemize}
\item[(1)] $\inf\{\frac{\|u-v\|^2}{\|u\|^2+\|v\|^2}:u\in H_1\ominus H_3, v\in H_2\ominus H_3\}$,
\item[(2)] $1-\|H_2H_1-H_3\|$,
\item[(3)] $1-\|H_1H_2H_1-H_3\|$.
\end{itemize}
\end{lem}
\begin{proof}
For the relation between the angle and (1), take $v$ by $-v$ in the previous equality. We get
$$
\frac{\|u-v\|^2}{\|u\|^2+\|v\|^2}=1-\frac{2Re\langle u,v\rangle}{\|u\|^2+\|v\|^2}.
$$
Therefore
\begin{eqnarray*}
&&\inf\{\frac{\|u-v\|^2}{\|u\|^2+\|v\|^2}:u\in H_1\ominus H_3, v\in H_2\ominus H_3\}\\
&=&1-\sup\{\frac{2Re\langle u,v\rangle}{\|u\|^2+\|v\|^2}:u\in H_1\ominus H_3, v\in H_2\ominus H_3\}\\
&=&1-\sup\{\frac{2|\langle u,v\rangle|}{\|u\|^2+\|v\|^2}:u\in H_1\ominus H_3, v\in H_2\ominus H_3\}
\end{eqnarray*}
Since
$$
\frac{|\langle u,v\rangle|}{\|u\|^2+\|v\|^2}\leq\frac{|\langle u,v\rangle|}{2\|u\|\|v\|}=\frac{|\langle au,a^{-1}v\rangle|}{\|au\|^2+\|a^{-1}v\|^2},
$$
where $a=\sqrt{\frac{\|v\|}{\|u\|}}$, we have
\begin{eqnarray*}
&&\inf\{\frac{\|u-v\|^2}{\|u\|^2+\|v\|^2}:u\in H_1\ominus H_3, v\in H_2\ominus H_3\}\\
&=&1-\sup\{\frac{|\langle u,v\rangle|}{\|u\|\|v\|}:u\in H_1\ominus H_3, v\in H_2\ominus H_3\}.
\end{eqnarray*}

Also, since $(H_2H_1-H_3)^*(H_2H_1-H_3)=H_1H_2H_1-H_3$, the quantities (2) and (3) are positively related. Now we show that (1) and (2) are positively related. Fix $u\in H_1\ominus H_3$, for any $v\in H_2\ominus H_3$,
\begin{eqnarray*}
&&\frac{\|u-v\|^2}{\|u\|^2+\|v\|^2}\geq \frac{\|u-v\|^2}{2\|u\|^2+\|u-v\|^2}\\
&\geq&\frac{\|u-H_2u\|^2}{2\|u\|^2+\|u-H_2u\|^2}\geq\frac{\|u-H_2u\|^2}{3\|u\|^2+3\|H_2u\|^2}.
\end{eqnarray*}
Also,
$$
\frac{\|u-H_2u\|^2}{\|u\|^2+\|H_2u\|^2}\approx\frac{\|u-H_2u\|^2}{\|u\|^2}.
$$
The first inequality shows that the infimum in (1) is obtained (modulo a constant) by taking $v=H_2u$. The second inequality shows that (1) and (2) are positively related. This completes the proof.
\end{proof}

From our discussion before Definition \ref{defnangle} and Lemma \ref{angle}(1), the following corollary is immediate.

\begin{cor}
$H_1+H_2$ is closed if and only if their angle is non-zero.
\end{cor}

As a consequence of Lemma \ref{angle}, when the projection operators to subspaces $H_1$, $H_2$ and $H_3$ all change continuously, the quantities in (3) change continuously, therefore the angles have a uniform lower bound on any compact set of parameters. This fact will be used in the proof of our main theorem.
\section{Decomposition of Varieties}
We begin this section with an example.
\begin{exam}\label{example hyperplane}
Suppose $\tilde{M}_1$ and $\tilde{M}_2$ are two linear subspaces of $\cn$. $\tilde{M}_3=\tilde{M}_1\cap \tilde{M}_2$. Let $M_i=\tilde{M}_i\cap\bn$ and $Q_i=\overline{span}\{K_{\lambda}| \lambda\in M_i\}\subseteq\ber$, $i=1,2,3$. Then
$$
\|Q_2Q_1Q_2-Q_3\|\leq a<1
$$
where the number $a$ depends on the angle between $\tilde{M}_1$ and $\tilde{M}_2$. As a consequence, $Q_1+Q_2$ is closed and $Q_1\cap Q_2=Q_3$.

\begin{proof}
First, to simplify notation, we use $Q_i$ and $M_i$ to denote both the spaces and the projection operators.

Let $\epsilon>0$ be determined later. Choose $k\in\mathbb{N}$ (depending on $\epsilon$) so that $\forall v\in M_2\ominus M_3$,
$$
|(M_2M_1)^kv|\leq\epsilon|v|.
$$

Clearly the operator $Q_2Q_1Q_2-Q_3$ vanishes on $Q_2^{\perp}$ and $Q_3$. For any $f\in Q_2\ominus Q_3$ with $\|f\|=1$, since
$$
(Q_2Q_1Q_2-Q_3)^kf=(Q_2Q_1)^kf,
$$
it suffices to prove that
$$
\|(Q_2Q_1)^kf\|\leq a^k.
$$
Let $d=\dim M_2$, by Example 3.3 in \cite{our paper}, the measure $\mu=c(1-|z|^2)^{n-d}dv_{M_2}$ with a suitable normalizing constant $c$  has the property that
$$\|Q_2 g\|^2=\int_{M_2}|g|^2d\mu, ~~~\forall g\in\ber.$$
Here $v_{M_2}$ is the volume measure on $M_2$.

It is easy to see that $Q_if(z)=f(M_i(z)), i=1,2,3$. Here we use $M_i$ to denote the projection operators to $\tilde{M}_i$. Now for any $z\in M_2$,
\begin{eqnarray*}
(Q_2Q_1)^kf(z)&=&Q_1(Q_2Q_1)^{k-1}f(z)\\
&=&(Q_2Q_1)^{k-1}f(M_1z)\\
&=&(Q_2Q_1)^{k-1}f(M_2M_1z)\\
&=&\cdots\\
&=&f((M_2M_1)^kz).
\end{eqnarray*}
\begin{eqnarray*}
(M_2M_1)^kz&=&(M_2M_1)^kM_3z+(M_2M_1)^k(1-M_3)z\\
&=&M_3z+(M_2M_1)^k(1-M_3)z.
\end{eqnarray*}

By the choice of $k$,
$$
|(M_2M_1)^k(1-M_3)z|^2\leq\epsilon^2|(1-M_3)z|^2\leq\epsilon^2(1-|M_3z|^2).
$$
Therefore the pseudo-hyperbolic metric
$$\rho((M_2M_1)^kz, M_3z)\leq r_{\epsilon},$$
where $r_{\epsilon}\to0$ when $\epsilon\to0$.

Before continuing, we need the following lemma.

\begin{lem}\label{continuity of function}
$\exists C>0$, $\forall g\in Hol(\mathbb{B}_d)$, $\forall z, w\in\mathbb{B}_d$, $\beta(z, w)<1/2$
$$
|g(z)-g(w)|^2\leq C\frac{\rho(z, w)^2}{(1-|w|^2)^{d+1}}\int_{D(w)}|g(\eta)|^2dv(\eta),
$$
where $D(w)=\{z | \beta(z,w)<1\}$.
\end{lem}

\begin{proof}
Using a reproducing kernel argument, it is easy to show that for $g\in Hol(\mathbb{B}_d)$ and $|\lambda|\in D(0,1/2)$,
$$
|g(\lambda)-g(0)|^2\leq C|\lambda|^2\int_{D(0)}|g(\eta)|^2dv(\eta).
$$
So if $\beta(z,w)<1/2$,
\begin{eqnarray*}
|g(z)-g(w)|^2&=&|g\varphi_w(\varphi_w(z))-g\varphi_w(0)|^2\\
&\leq&C|\varphi_w(z)|^2\int_{D(0)}|g\varphi_w(\eta)|^2dv(\eta)\\
&=&C\rho(z,w)^2\int_{D(w)}|g(\lambda)|^2\frac{(1-|w|^2)^{d+1}}{|1-\langle \lambda,w\rangle|^{2(d+1)}}dv(\lambda)\\
&\leq&C\frac{\rho(z,w)^2}{(1-|w|^2)^{d+1}}\int_{D(w)}|g(\lambda)|^2dv(\lambda)
\end{eqnarray*}
This completes the proof of lemma.
\end{proof}

From the Lemma and previous argument,
\begin{eqnarray*}
&&|f((M_2M_1)^kz)|\\
&=&|f((M_2M_1)^kz)-f(M_3z)|\\
&\leq&Cr_{\epsilon}\frac{1}{(1-|M_3z|^2)^{d+1}}\int_{D(M_3z)}|f(\eta)|^2dv_{M_2}(\eta)(1-|z|^2)^{n-d}dv_{M_2}(z).
\end{eqnarray*}
Therefore
\begin{eqnarray*}
&&\|(Q_2Q_1)^kf\|^2=\int_{M_2}|(Q_2Q_1)^kf(z)|^2(1-|z|^2)^{n-d}dv_{M_2}(z)\\
&=&\int_{M_2}|f((M_2M_1)^kz)|^2(1-|z|^2)^{n-d}dv_{M_2}(z)\\
&\leq&Cr_{\epsilon}^2\int_{M_2}\frac{1}{(1-|M_3z|^2)^{d+1}}\int_{D(M_3z)}|f(\eta)|^2dv_{M_2}(\eta)(1-|z|^2)^{n-d}dv_{M_2}(z)\\
&=&Cr_{\epsilon}^2\int_{M_2}|f(\eta)|^2\int_{\{z\in M_2: M_3z\in D(\eta)\}}\frac{(1-|z|^2)^{n-d}}{(1-|M_3z|^2)^{d+1}}dv_{M_2}(z)dv_{M_2}(\eta).
\end{eqnarray*}
If $ M_3z\in D(\eta)$, then $\beta(M_3z,M_3\eta)\leq\beta(M_3z,\eta)<1$ and therefore
$$
\beta(M_3\eta,\eta)\leq\beta(M_3z,M_3\eta)+\beta(M_3z,\eta)<2.
$$
Thus $1-|M_3z|^2\approx 1-|M_3\eta|^2\approx 1-|\eta|^2$. We claim that
$$
\int_{\{z\in M_2: M_3z\in D(\eta)\}}(1-|z|^2)^{n-d}dv_{M_2}(z)\leq C(1-|\eta|^2)^{n+1}.
$$
For $z\in M_2$, write temporarily $z=(z',z'')$ where $z'$ corresponds to the coordinates in $M_3$.
\begin{eqnarray*}
&&\int_{\{z\in M_2: M_3z\in D(\eta)\}}(1-|z|^2)^{n-d}dv_{M_2}(z)\\
&=&C\int_{z'\in D(\eta)}(1-|z'|^2)^{n-d+d-d_3}\int_{\lambda\in\mathbb{B}_{d-d_3}}(1-|\lambda|^2)^{n-d}dv(\lambda)dv(z')\\
&\leq&C(1-|\eta|^2)^{n-d_3}(1-|\eta|^2)^{d_3+1}\\
&=&C(1-|\eta|^2)^{n+1}.
\end{eqnarray*}
This proves the claim.

Therefore
$$
\|(Q_2Q_1)^kf\|^2\leq Cr_{\epsilon}^2\|f\|^2.
$$
Take $\epsilon>0$ such that $Cr_{\epsilon}^2<1$ in the beginning, then our proof is complete.
\end{proof}
\end{exam}
\begin{rem}\label{inverse of exam 1}
From Lemma \ref{continuity of function}, it is easy to see that
$$
\|k_z-k_w\|\leq C\rho(z,w)
$$
when $\rho(z,w)$ is small. This tells us that the inverse of Example \ref{example hyperplane} is also true: if the angle between $M_1$ and $M_2$ is small, then so is the angle between $Q_1$ and $Q_2$. Take $z_1\in M_1$, $z_2\in M_2$ such that $z_i\perp M_3$, then $Q_3k_{z_i}=0, i=1,2$ and $\|k_{z_1}-k_{z_2}\|\leq C\rho(z_1,z_2)$. When the angle of $M_1$ and $M_2$ is small we can take such $z_i$ so that $\rho(z_1,z_2)$ is small. Therefore by Lemma \ref{angle}, the angle between $Q_1$ and $Q_2$ is small.
\end{rem}

\begin{exam}\label{counter example}
Suppose $\tilde{M}_1$ and $\tilde{M}_2$ are two affine spaces, $\emptyset\neq \tilde{M}_1\cap \tilde{M}_2\cap\overline{\mathbb{B}_n}\subseteq\partial\mathbb{B}_n$. Let $M_i=\tilde{M}_i\cap\bn$, $Q_i=\overline{span}\{K_{\lambda}:\lambda\in M_i\}$. Then $Q_1\cap Q_2=\{0\}$ and $Q_1+Q_2$ is not closed.
\end{exam}

\begin{proof}
Since $Q_1\cap Q_2$ is the orthogonal space of a polynomial ideal with generators of degree one, and has no zero points inside $\bn$ and only one on $\partial\mathbb{B}_n$, $Q_1\cap Q_2=\{0\}$, we have $Q_1\cap Q_2=\{0\}$.

For $z\in\mathbb{B}_n$, it is easy to prove that $Q_{iz}$ is the projection to the space $\overline{span}\{K_{\lambda}:\lambda\in\varphi_z(M_i)\}$. Therefore, without loss of generality, we assume that $M_1$ is a linear subspace.

We claim that $\rho(M_1,M_2)=0$. To prove this, take $z\in M_1\cap M_2\cap\partial\mathbb{B}_n$, then $rz\in M_1$, $\forall 0<r<1$. Change coordinates so that $z=(z_1,0,\ldots,0)$. Since $M_2$ is an affine space that intersects $\mathbb{B}_n$, after possibly changing the order of basis, $M_2$ has expression
$$
w=(w',L(w'))+(z_1,0,\ldots,0), w\in M_2,
$$
where $L$ is a linear function of $w'=(w_1,\ldots,w_d)$, $d=\dim M_2$. Take $w_r=((r-1)z',L((r-1)z'))+z$, then
$$
\varphi_{rz}(w_r)=(\varphi_{rz'}(rz'),\frac{(1-r^2)^{1/2}}{1-r^2}\cdot(r-1)L(z'))=(0,O((1-r^2)^{1/2})).
$$
Therefore $\rho(rz,w_r)\to0, r\to1$. This proves the claim. The rest of the proof is same as in Remark \ref{inverse of exam 1}.
\end{proof}

Next we discuss the more general case: suppose $M_1$ and $M_2$ are two varieties and $M_3=M_1\cap M_2$. Let $Q_i=\overline{span}\{K_{\lambda}:\lambda\in M_i\}$, $1=1,2,3$. When do we know that $Q_1+Q_2$ is closed and $Q_1\cap Q_2/Q_3$ is finite dimensional? This question is important because when it holds, the essential normality of $Q=Q_1+Q_2$ follows from the essential normality of each $Q_1$ and $Q_2$. Moreover, we can obtain the index class of $Q_1+Q_2$ from the index classes of $Q_1$, $Q_2$ and $Q_3$. This would allow us to obtain index results for complicated varieties by decomposing it into nice pieces.

As preparation for our main result, we establish a few lemmas.
\begin{lem}\label{rhoxy}
Suppose $x$, $y\in M_A\backslash \mathbb{B}_n$ and $\rho(x,y)<1$, then there exists a unitary operator $U$ such that for any $S\in\mathcal{T}(L^{\infty})$,
$$
S_y=U^*S_xU.
$$
\end{lem}

\begin{proof}
Suppose $z_{\alpha}\to x$, $\{z_{\alpha}\}\subseteq\bn$. Since $\rho(x,y)<1$, by Lemma \ref{extended rho}, $\exists\lambda\in\mathbb{B}_n$ such that $w_{\alpha}:=\varphi_{z_{\alpha}}(\lambda)\to y$.

\begin{eqnarray*}
U_{w_{\alpha}}f(z)&=&f\circ\varphi_{w_{\alpha}}(z)\frac{(1-|w_{\alpha}|^2)^{(n+1)/2}}{(1-\langle z,w_{\alpha}\rangle)^{n+1}}\\
&=&f\circ U_{\alpha}\circ\varphi_{\lambda}\circ\varphi_{z_{\alpha}}(z)a_{\alpha}\frac{(1-|z_{\alpha}|^2)^{(n+1)/2}}{(1-\langle z,z_{\alpha}\rangle)^{n+1}}\frac{(1-|\lambda|^2)^{(n+1)/2}}{(1-\langle\varphi_{z_{\alpha}}(z),\lambda\rangle)^{n+1}}\\
&=&a_{\alpha}U_{z_{\alpha}}U_{\lambda}(f\circ U_{\alpha})(z).
\end{eqnarray*}
Here
$$
a_{\alpha}=\frac{(1-\langle z_{\alpha},\lambda\rangle)^{n+1}}{|1-\langle z_{\alpha},\lambda\rangle|^{n+1}}
$$
is a number of absolute value $1$ and $U_{\alpha}$ is a unitary operator on $\mathbb{C}^n$ such that $\varphi_{w_{\alpha}}=U_{\alpha}\circ\varphi_{\lambda}\circ\varphi_{z_{\alpha}}$(see the proof of Lemma 6.2 in \cite{Suarez07} for existence of such $U_{\alpha}$).
Now we can take a subnet such that $U_{\alpha}\to U'$. Here $U'$ is a unitary operator on $\mathbb{C}^n$.

Therefore, for $f, g\in\mathbb{C}[z_1,\ldots,z_n]$,
\begin{eqnarray*}
\langle S_{w_{\alpha}}f,g\rangle&=&\langle SU_{w_{\alpha}}f,U_{w_{\alpha}}g\rangle\\
&=&\langle SU_{z_{\alpha}}U_{\lambda}(f\circ U_{\alpha}),U_{z_{\alpha}}U_{\lambda}(g\circ U_{\alpha})\rangle\\
&=&\langle S_{z_{\alpha}}U_{\lambda}(f\circ U_{\alpha}),U_{\lambda}(g\circ U_{\alpha})\rangle\\
&=&\langle S_{z_{\alpha}}U_{\lambda}(f\circ U_{\alpha}-f\circ U'),U_{\lambda}(g\circ U_{\alpha}))\rangle\\
& &+\langle S_{z_{\alpha}}U_{\lambda}(f\circ U'),U_{\lambda}(g\circ U_{\alpha}-g\circ U')\rangle\\
& &+\langle S_{z_{\alpha}}U_{\lambda}(f\circ U'),U_{\lambda}(g\circ U')\rangle.
\end{eqnarray*}
Note that $f\circ U_{\alpha}$ tends to $f\circ U'$ in norm, and that $\|S_{\alpha}\|\leq\|S\|$, by standard argument, the first two terms converge to $0$. Define $Uf=U_{\lambda}(f\circ U')$. Taking limit, we see that for any polynomial $f, g$,
$$
\langle S_yf,g\rangle=\langle S_xUf,Ug\rangle.
$$
Therefore $S_y=U^*S_xU$. This completes the proof.
\end{proof}

\begin{lem}\label{3equivalent}
Suppose $Q_1$, $Q_2$, $Q_3$ are closed linear subspaces of $L_{a}^2(\mathbb{B}_n)$, $Q_3\subseteq Q_1\cap Q_2$, the projection operators $Q_i\in\mathcal{T}(L^{\infty})$. Then the following are equivalent.
\begin{itemize}
\item[(1)] $Q_1+Q_2$ is closed and $Q_1\cap Q_2/Q_3$ is finite dimensional.
\item[(2)] $\|Q_2Q_1Q_2-Q_3\|_e<1$.
\item[(3)] $\exists 0<a<1$ such that $\forall x\in M_A\backslash \mathbb{B}_n$,
$$
\|Q_{2x}Q_{1x}Q_{2x}-Q_{3x}\|<a.
$$
\end{itemize}
\end{lem}
\begin{proof}
The equivalence of (1) and (2) can be obtained by analysing the spectral decomposition of the operator $Q_2Q_1Q_2-Q_3$. And the equivalence of (3) and (2) is by Lemma \ref{essential norm}.
\end{proof}

\begin{lem}\label{Mx}
Suppose $\tilde{M}$ is an analytic subset of an open neighborhood of $\overline{\mathbb{B}_n}$. $\tilde{M}$ is smooth on $\partial\mathbb{B}_n$ and transversal with $\partial\mathbb{B}_n$. Let $M=\tilde{M}\cap\mathbb{B}_n$ and
$$Q=\overline{span}\{K_{\lambda}|\lambda\in M\}.$$
Then for $x\in M_A\backslash \mathbb{B}_n$.

If $\rho(x,\overline{M}^{_A})=1$, then $Q_x=0$;

If $x\in\overline{M}^{_A}$, then $Q_x=\overline{span}\{K_{\lambda}|\lambda\in M_x\}$, where $M_x=\tilde{M}_x\cap\bn$ and
$$
\tilde{M}_x=\{v\in T\tilde{M}|_{\hat{x}}: v\perp\hat{x}\}+\mathbb{C}\hat{x}.
$$
Where $\hat{x}\in\pbn$ is obtained by evaluating $x$ at each index function $z_i$.

If $\rho(x,\overline{M}^{_A})<1$, then $\exists y\in\overline{M}^{_A}$ such that $Q_x$ is unitary equivalent to  $Q_y$.

Here $\overline{M}^{_A}$ denotes the closure of $M$ in $M_A$.
\end{lem}

\begin{proof}
By the proof in \cite{our paper}, there exists an ``equivalent measure'' $\mu$ on $M$ such that $0$ is isolated in $\sigma(T_{\mu})$ and $Q=Ran T_{\mu}$. In other words, if we take a continuous function $f$ on $\mathbb{R}$ such that $f(0)=0$ and $f$ takes value $1$ on $\sigma(T_{\mu}\backslash\{0\})$, then the projection operator  $Q=f(T_{\mu})$. Therefore $Q_x=f((T_{\mu})_x)$. Suppose $z_{\alpha}\to x$, $z_{\alpha}\in\mathbb{B}_n$. The operators $(T_{\mu})_{z_{\alpha}}$ tend to $(T_{\mu})_x$ in the strong operator topology. Since
$$
(T_{\mu})_{z_{\alpha}}=T_{\mu_{z_{\alpha}}},
$$
where the positive measure $\mu_{z_{\alpha}}$ is defined by
$$
\int gd\mu_{z_{\alpha}}=\int g\circ\varphi_{z_{\alpha}}|k_{z_{\alpha}}|^2d\mu.
$$
From the definition,
$$
\|\mu_{z_{\alpha}}\|=\int d\mu_{z_{\alpha}}=\int |k_{z_{\alpha}}|^2d\mu\leq C\|k_{z_{\alpha}}\|=C
$$
since $\mu$ is a Carleson measure. Therefore
$$
\|\mu_{z_{\alpha}}\|\leq C.
$$
Therefore the net $\{\mu_{z_{\alpha}}\}$ has a subnet that converges to some measure $\mu_x$ in the $weak^*$ topology. Then
$$
\langle(T_{\mu})_xg,h\rangle=\int g\bar{h}d\mu_x, \forall g,h\in \mathbb{C}[z_1,\ldots,z_n].
$$
So $\mu_x$ is a Carleson measure and $(T_{\mu})_x=T_{\mu_x}$.

From our construction, we see that $Q_x$ is the projection operator onto $Range T_{\mu_x}=\overline{span}\{K_{\lambda}|\lambda\in supp\mu_x\}$.

Next we discuss about $supp\mu_x$. We claim that
\begin{eqnarray*}
supp\mu_x&=&M_x:=\{w\in\mathbb{B}_n| \rho(\varphi_{z_{\alpha}}(w), M)\to0\}\\
&=&\{w\in\mathbb{B}_n| \rho(w,\varphi_{z_{\alpha}}(M))\to0\}.
\end{eqnarray*}
For any $w\in\mathbb{B}_n$ and $0<r<1$,
\begin{eqnarray*}
\mu_{z_{\alpha}}(D(w,r))&=&\int\chi_{D(w,r)}d\mu_{z_{\alpha}}\\
&=&\int\chi_{D(\varphi_{z_{\alpha}}(w),r)}|k_{z_{\alpha}}|^2d\mu\\
&\approx&\mu(D(\varphi_{z_{\alpha}}(w),r))\frac{|1-\langle w,z_{\alpha}\rangle|^{2(n+1)}}{(1-|z_{\alpha}|^2)^{n+1}},
\end{eqnarray*}
where the last inequality is because for $\lambda\in D(\varphi_{z_{\alpha}}(w),r)$,
\begin{eqnarray*}
|k_{z_{\alpha}}(\lambda)|^2&=&\frac{(1-|z_{\alpha}|^2)^{n+1}}{|1-\langle\lambda,z_{\alpha}\rangle|^{2(n+1)}}\\
&\approx&\frac{(1-|z_{\alpha}|^2)^{n+1}}{|1-\langle\varphi_{z_{\alpha}}(w),z_{\alpha}\rangle|^{2(n+1)}}\\
&=&\frac{|1-\langle w,z_{\alpha}\rangle|^{2(n+1)}}{(1-|z_{\alpha}|^2)^{n+1}}.
\end{eqnarray*}
Then if $\rho(\varphi_{z_{\alpha}}(w),M)\to0$, for any $0<r<1$, we have $$\mu(D(\varphi_{z_{\alpha}}(w),r))\approx(1-|\varphi_{z_{\alpha}}(w)|^2)^{n+1}.$$ So
\begin{eqnarray*}
\mu_{z_{\alpha}}(D(w,r))&\approx&(1-|\varphi_{z_{\alpha}}(w)|^2)^{n+1}\frac{|1-\langle w,z_{\alpha}\rangle|^{2(n+1)}}{(1-|z_{\alpha}|^2)^{n+1}}\\
&=&(1-|w|^2)^{n+1}.
\end{eqnarray*}
Therefore $\mu_x(D(w,r))>0$ for any $0<r<1$, i.e., $w\in supp\mu_x$.

On the other hand, if $\rho(\varphi_{z_{\alpha}}(w), M)\nrightarrow0$, by taking a subnet we can assume that $\rho(\varphi_{z_{\alpha}}(w),M)>\epsilon>0$. Take $r<\epsilon$ in the proof and it is easy to see that $\mu_x(D(w,r))=0$. Therefore $w$ is not in the support. This completes the proof of our claim.

Now we study the set $M_x$.

First, suppose $\rho(x,\overline{M}^{_A})=1$. By definition, this means that for any $0<r<1$, there exists a net $z_{\alpha}\to x$, such that $\rho(\{z_{\alpha}\}, M)>r$. Therefore for any $0<r'<1$, choose $r>r'$, then from the proof above it is easy to see that $\mu_x(D(0,r'))=0$, which implies $\mu_x=0$. Therefore $Q_x=0$.

Second, the case $\rho(x,\overline{M}^{_A})<1$ is by Lemma \ref{rhoxy}.

Finally, when $x\in\overline{M}^{_A}$, suppose $\tilde{M}$ has local expression $$
w=(w_1,\ldots,w_d,F_{d+1}(w'),\ldots,F_n(w')), w\in\tilde{M}\cap B(z,\delta),
$$
where $w'=(w_1,\ldots,w_d)$. Note that we are using the same kind of expression as in \cite{our paper}, where the basis and functions change continuous with $z$. And the point $z$ always has expression $(z_1,0,\ldots,0)$.

Suppose $w\in M_x$ and suppose $z_{\alpha}\in M$, $z_{\alpha}\to x$. Then by definition $\exists\lambda_{\alpha}\in M$ such that $$\rho(w,\varphi_{z_{\alpha}}(\lambda_{\alpha}))\to0.$$
which is equivalent to $|w-\varphi_{z_{\alpha}}(\lambda_{\alpha})|\to0$.
Take any $\epsilon>0$ such that $|w|+\epsilon<1$, then for some subnet we have $\rho(\lambda_{\alpha},z_{\alpha})=|\varphi_{z_{\alpha}}(\lambda_{\alpha})|<|w|+\epsilon<1$. Therefore $$1-\langle\lambda_{\alpha},z_{\alpha}\rangle\approx1-|z_{\alpha}|^2.$$
Since $\lambda_{\alpha}=(\lambda_{\alpha}',F_{\alpha}(\lambda_{\alpha}'))$
under the basis determined by $z_{\alpha}$.
$$
\varphi_{z_{\alpha}}(\lambda_{\alpha})=(\eta_{\alpha,1},\eta_{\alpha,2},\ldots,\eta_{\alpha,d},\ldots,\eta_{\alpha,n})
$$
where
$$
\eta_{\alpha,1}=\frac{z_{\alpha,1}-\lambda_{\alpha,1}}{1-\langle\lambda_{\alpha},z_{\alpha}\rangle},~~\eta_{\alpha,i}=-\frac{(1-|z_{\alpha}|^2)^{1/2}}{1-\langle\lambda_{\alpha},z_{\alpha}\rangle}\lambda_{\alpha,i},~~i=2,\ldots,d.
$$
and
$$
\eta_{\alpha,i}=-\frac{(1-|z_{\alpha}|^2)^{1/2}}{1-\langle\lambda_{\alpha},z_{\alpha}\rangle}F_{\alpha,i}(\lambda_{\alpha}'),~~i=d+1,\ldots,n.
$$
For simplicity we omit the subscript $\alpha$.

Now $$F_i(\lambda')=L_i(\lambda')+O(|z-\lambda|^2)=L_i(\lambda')+O(1-|z|^2).$$
Here $L_i$ is the linear part of $F_i$:
$$
L_i(\lambda')=\sum_{j=1}^dA_j(\lambda_j-z_j)=A_1(\lambda_1-z_1)+\sum_{j=2}^dA_j\lambda_j.
$$
Then
$$
\eta_i=(1-|z|^2)^{1/2}A_1\eta_1+\sum_{j=2}^dA_j\eta_j+O((1-|z|^2)^{1/2}),~~~j=d+1,\ldots,n.
$$
Since $\eta\to w$ as $z\to x$ and the coefficients $A_i$ converges to the corresponding value at $\hat{x}$. We see that
$$
w_i=\sum_{j=2}^dA_jw_j, i=d+1,\ldots,n.
$$
Also if $w$ is of this form, the argument above also implies that $w\in M_x$.

To write more explicitly, $M_x=\tilde{M}_x\cap\bn$ and
$$
\tilde{M}_x=\{v\in T\tilde{M}|_{\hat{x}}: v\perp\hat{x}\}+\mathbb{C}\hat{x}.
$$
\end{proof}

Now suppose $M_1$ and $M_2$ are as in Lemma \ref{Mx}. Let $M_3=M_1\cap M_2$ and let $Q_i=\overline{span}\{K_{\lambda}|\lambda\in M_i\}$. We want to find a suitable condition to ensure that $Q_1+Q_2$ is closed and $Q_1\cap Q_2/Q_3$ is finite dimensional. An equivalent condition is that $\|Q_2Q_1Q_2-Q_3\|_e<1$. From Theorem \ref{one variety}, the projections $Q_1$ and $Q_2$ are already in $\mathcal{T}(L^{\infty})$. Assume $M_3$ is also as in Lemma \ref{Mx}, then by Lemma \ref{3equivalent}, we only need to look at the operators $Q_{2x}Q_{1x}Q_{2x}-Q_{3x}$, $x\in M_A\backslash\mathbb{B}_n$. From Lemma \ref{Mx} we know $Q_{1x}$, $Q_{2x}$ and $Q_{3x}$ are projections to quotient modules corresponding to linear varieties $M_{ix}$. We list the cases that are possible:
\begin{itemize}
\item[(1)]$\rho(x,\overline{M_1}^{_A})<1$, $\rho(x, \overline{M_2}^{_A})<1$. In this case, there are two possibilities for $M_3$:
    \begin{itemize}
\item[(1-a)] $\rho(x,\overline{M_3}^{_A})<1$ or
\item[(1-b)]$\rho(x,\overline{M_3}^{_A})=1$.
\end{itemize}
\item[(2)]$\rho(x,\overline{M_1}^{_A})=1$, $\rho(x,\overline{M_2}^{_A})<1$, then $\rho(x,\overline{M_3}^{_A})=1$.
\item[(3)]$\rho(x,\overline{M_1}^{_A})<1$, $\rho(x,\overline{M_2}^{_A})=1$, then $\rho(x,\overline{M_3}^{_A})=1$.
\item[(4)]$\rho(x,\overline{M_1}^{_A})=1$, $\rho(x,\overline{M_2}^{_A})=1$, then $\rho(x,\overline{M_3}^{_A})=1$.
\end{itemize}

The case (1-b) corresponds to that of Example \ref{counter example}, which we want to avoid.

Under the cases (2)(3)(4), the operator $Q_{2x}Q_{1x}Q_{2x}-Q_{3x}=0$. In the case (1-a), we can assume $x\in\overline{M_3}^{_A}$, then it becomes Example \ref{example hyperplane}. The only thing that matters is the angle between the two linear subspaces $M_{1x}$ and $M_{2x}$ which are defined in Lemma \ref{Mx}.

To summarize, we need the following conditions:
\begin{itemize}
\item[(1)] $M_1$ and $M_2$ are transverse with $\pbn$, smooth on $\partial\mathbb{B}_n$.
\item[(2)] $M_3$ is also transverse with $\pbn$, smooth on $\partial\mathbb{B}_n$.
\item[(3)] If $x\in M_A\backslash \mathbb{B}_n$ and $\rho(x, \overline{M_1}^{_A})<1$, $\rho(x,\overline{M_2}^{_A})<1$, then $\rho(x,\overline{M_3}^{_A})<1$.
\item[(4)] $\forall x\in\overline{M_3}^{_A}\backslash\mathbb{B}_n$, $M_{1x}\cap M_{2x}=M_{3x}$.
\item[(5)] For all $x\in\overline{M_3}^{_A}\backslash\mathbb{B}_n$, the angles between $M_{1x}$ and $M_{2x}$ have a lower bound.
\end{itemize}

Next we seek properties that would ensure condition (1)-(5). Let us begin with a definition.
\begin{defn}
Let $K$ and $L$ be embedded submanifolds of a manifold $M$ and suppose that their intersection $K\cap L$ is also an embedded submanifold of $M$. Then $K\cap L$ are said to have \emph{clean intersection} if for each $p\in K\cap L$ we have
$$
T_p(K\cap L)= T_pK\cap T_pL.
$$
\end{defn}

Condition (1) and (2) must be assumed so that each of the three quotient modules alone are essentially normal. Assume further that $M_1$ and $M_2$ intersect cleanly at each point of intersection in $\partial\mathbb{B}_n$, then it is not hard to verify (1)(2)(4)(5). Condition (4) follows from the definition of clean intersection and the expression we obtained for $M_x$ in Lemma \ref{Mx}. The fact that $M_i$ is smooth on $\pbn$ tells us that the projection operator to $M_{ix}$ depend continuously on $x$, $i=1,2,3$. Therefore condition (5) follows from Lemma \ref{angle} and compactness. The fact that the assumptions above also ensure condition (3) takes some effort to prove. Assuming this, we have reached our main theorem of this paper.

\begin{thm}\label{mainthm}
Suppose $\tilde{M_1}$ and $\tilde{M_2}$ are two analytic subsets of an open neighborhood of $\clb$. Let $\tilde{M_3}=\tilde{M_1}\cap\tilde{M_2}$. Assume that
\begin{itemize}
\item[(i)]$\tilde{M_1}$ and $\tilde{M_2}$ intersect transversely with $\pbn$ and have no singular points on $\pbn$.
\item[(ii)]$\tilde{M_3}$ also intersects transversely with $\pbn$ and has no singular points on $\pbn$.
\item[(iii)]$\tilde{M_1}$ and $\tilde{M_2}$ intersect cleanly on $\partial\mathbb{B}_n$.
\end{itemize}
Let $M_i=\tilde{M_i}\cap\mathbb{B}_n$ and $Q_i=\overline{span}\{K_{\lambda}: \lambda\in M_i\}$, $i=1,2,3.$ Then $Q_1\cap Q_2/Q_3$ is finite dimensional and $Q_1+Q_2$ is closed. As a consequence, $Q_1+Q_2$ is $p$-essentially normal for $p>2d$, where $d=\max\{\dim M_1, \dim M_2\}$.
\end{thm}

As stated above the theorem, the only thing left for us to verify is the following lemma.

\begin{lem}\label{condition 3}
Assume the same conditions as Theorem \ref{mainthm}, then the technical condition (3) holds.
\end{lem}

We break the proof into several lemmas.

\begin{lem}\label{condition 3 lemma 1}
Suppose $z=(z_1,0,\cdots,0)\in\bn$, then
$$
\frac{\partial}{\partial w_1}|\varphi_z(w)|^2(0)=\bar{z_1}(|z_1|^2-1)
$$
and
$$
\frac{\partial}{\partial w_i}|\varphi_z(w)|^2(0)=0,~~~i=2,3,\cdots,n.
$$
As a consequence, if $M$ is any complex manifold passing through $0$ and obtains its minimal hyperbolic distant to $z$ at the point $0$, then $z$ must be orthogonal to the tangent space $TM|_0$.
\end{lem}

\begin{proof}
The two formulas are obtained by direct computation. To prove the last statement, one only need to observe that the derivative of $|\varphi_z(w)|^2$ in $u$ direction is
$$
\frac{\partial|\varphi_z(w)|^2}{\partial u}(0)=\sum_{i=1}^nu_i\frac{\partial|\varphi_z(w)|^2}{\partial w_i}(0)=\langle u,z\rangle(|z_1|^2-1).
$$
Since the minimal value of $|\varphi_z(w)|^2$ is obtained at $0$, the derivative of $|\varphi_z(w)|^2$ along all directions in $TM|_0$ must be $0$. Therefore $z$ is orthogonal to $TM|_0$. This completes the proof.
\end{proof}

\begin{lem}\label{condition 3 lemma 2}
Suppose $M$ satisfies the hypotheses of Theorem \ref{one variety}, and suppose $\{z_{\alpha}\}, \{w_{\alpha}\}\subseteq M$ are two separated nets such that , viewed as points in $M_A$, $z_{\alpha}$ tends to a point $x\in M_A\backslash\bn$, viewed as points in $\clb$, $w_{\alpha}$ tends to $\hat{x}$. Then any limit point of the net $\{\varphi_{z_{\alpha}}(w_{\alpha})\}$ is in $\overline{M_x}\subseteq\clb$.
\end{lem}

\begin{proof}
For convenience we omit the subscript $\alpha$. Using the same convention as before, we take the basis at each $z$, so
$$
z=(z_1,0,\cdots,0),~~~w=(w',F(w')),
$$
where $w'=(w_1,\cdots,w_d)$ and $F=(f_{d+1},\cdots,F_n)$ is the expression of $\tilde{M}$ depending continuously on $z$. Same as in the proof of Lemma \ref{Mx}, we have
$$
\varphi_z(w)=(\eta_1,\cdots,\eta_n),
$$
where
$$
\eta_1=\frac{z_1-w_1}{1-\langle w,z\rangle},~~~\eta_i=-\frac{(1-|z|^2)^{1/2}w_i}{1-\langle w,z\rangle},i=2,\cdots,d
$$
and
$$
\eta_j=-\frac{(1-|z|^2)^{1/2}F_j(w')}{1-\langle w,z\rangle},j=d+1,\cdots,n.
$$
We write $F_j(w')=L_j(w')+O(|w-z|^2)$, where $L$ is the linear part of $F$.
Since
$$
|w-z|^2=|w|^2+|z|^2-2Re\langle w,z\rangle\leq2(1-Re\langle w,z\rangle)\leq2|1-\langle w,z\rangle|,
$$
for $j=d+1,\cdots,n$,
$$
\eta_j+\frac{(1-|z|^2)^{1/2}L_j(w')}{1-\langle w,z\rangle}=\frac{(1-|z|^2)^{1/2}}{1-\langle z,w\rangle}O(|1-\langle w,z\rangle|)\to0,~~z\to\hat{x}.
$$
The rest of the proof is as in Lemma \ref{Mx}(3).
\end{proof}

\begin{proof}[\textbf{Proof of Lemma \ref{condition 3}}]
Suppose $x\in M_A\backslash\bn$ and $\rho(x,\overline{M_1}^{_A})<1$, $\rho(x,\overline{M_2}^{_A})<1$, we will show that $\rho(x,\overline{M_3}^{_A})<1$. Clearly, $\hat{x}\in\tilde{M}_1\cap\tilde{M}_2=\tilde{M}_3$. By Lemma \ref{Mx}, without loss of generality, we assume $x\in\overline{M_1}^{_A}$.

Let $\{z_{\alpha}\}\subseteq M_1$ such that $z_{\alpha}\to x$. Let $w_{\alpha}\in M_2$ and $\lambda_{\alpha}\in M_3$ such that $\rho(z_{\alpha},w_{\alpha})=\rho(z_{\alpha},M_2)$ and $\rho(z_{\alpha},\lambda_{\alpha})=\rho(z_{\alpha},M_3)$. Take subnets (using the same notation) such that both nets converge in $M_A$. Suppose $w_{\alpha}\to y\in M_A$ and $\lambda_{\alpha}\to \xi\in M_A$. Clearly $\hat{y}=\hat{\xi}=\hat{x}$. For convenience we omit the subscript $\alpha$ in the sequel.

Since $\rho(\varphi_{\lambda}(z),0)=\rho(\varphi_{\lambda}(z),\varphi_{\lambda}(M_3))$, by Lemma \ref{condition 3 lemma 1}, $\varphi_{\lambda}(z)\perp\varphi_{\lambda}(M_3)$. The latter tends uniformly to $M_{3\xi}$ while the first has a subnet that converges to some point $a$ in $\pbn$ by compactness. Therefore $a\perp M_{3\xi}$.

On the other hand, $\rho(\varphi_{\lambda}(z),\varphi_{\lambda}(w))=\rho(z,w)\to\rho(x,\overline{M_2}^{_A})<1$. Since $|\varphi_{\lambda}(z)|=\rho(\lambda,z)\to1$, we have the Euclidean distance $|\varphi_{\lambda}(z)-\varphi_{\lambda}(w)|\to0$. Therefore $\varphi_{\lambda}(w)\to a$. By Lemma \ref{condition 3 lemma 2}, $a\in M_{1\xi}\cap M_{2\xi}$ which equals $M_{3\xi}$ by the clean intersection condition and the experession of $M_{i\xi}$ in Lemma \ref{Mx}. So $a$ is a vector of length $1$ which both belong to $M_{3\xi}$ and is perpendicular to $M_{3\xi}$. A contradiction. Therefore such $x$ does not exist. This completes the proof.

\end{proof}

\section{Summary}
A classical way of proving Geometric Arveson-Douglas Conjecture is by ``decomposing the variety'' (cf. \cite{Sha Ken}). In this paper, we introduce ideas in complex harmonic analysis to solve this problem. This approach has the advantage of ``localizing'' the problem, which allows us to reduce the problem to simpler cases. Su\'{a}rez's results (\cite{Suarez04}\cite{Suarez07}) play an important role here.

The ideas in our last paper \cite{our paper} and this paper should be considered as two continuous steps towards analysing the varieties. First, we approximate the variety at points close to $\pbn$, using simpler varieties (in our case, their linearizations at the points). Then we use results on these simpler varieties to obtain essential normality results of the original variety. After the first step, we are able to ``localize'' the variety at points in $M_A\backslash \bn$. We then obtain results on relation between the angle of two quotient modules and the relative positions of ``localizations'' of the corresponding varieties. Finally, we give sufficient conditions for the sum of two quotient modules to be closed, i.e., the angle to be positive. This gives us results on unions of varieties. In the future, when we proved more results in the first step, we can use similar techniques in this paper to generate more complicated examples.

Another consequence of our result is an index result. Given $Q_i$ as in Theorem \ref{mainthm}, consider the exact sequence
$$
0\to Q_1\cap Q_2\to Q_1\oplus Q_2\to Q_1+Q_2\to0.
$$
Here we define the first map to be the embedding and second map to be the difference of two entries. In general, given such a short exact sequence and given that the sum on the right side is closed, then the essential normality of the two modules imply the essential normality of both their sum and their intersection (cf. \cite{Douglas Wang remark}).

 Also, By BDF theory, the essentially normal quotient modules $Q_i$, $i=1,2,3$ and $Q=Q_1+Q_2$ define index classes, or elements in $K_1(\tilde{M_i}\cap\pbn)$ and $K_1((\tilde{M_1}\cup\tilde{M_2})\cap\pbn)$, respectively. Since $Q_1\cap Q_2/Q_3$ is finite dimensional, the index class $[Q_1\cap Q_2]=[Q_3]$. Therefore we have an equation of index classes
 $$[Q_3]+[Q]=[Q_1]+[Q_2].$$
  In particular, if we assume further that the varieties $M_i$, $i=1,2,3$ satisfy the assumptions of \cite{DYT}, then the index results in \cite{DYT} apply to $[Q_i]$ and we get a formula for $[Q]$ from the above equation.

Ronald G.~Douglas

Texas A\&M University, College Station, TX, 77843

E-mail address: rdouglas@math.tamu.edu\\

Yi Wang

Texas A\&M University, College Station, TX, 77843

E-mail address: yiwangfdu@gmail.com
\end{document}